\documentclass[12pt]{amsart}

\usepackage{amsfonts, amsmath, amsthm, enumerate}

\input xy
\xyoption{matrix}\xyoption{arrow}\xyoption{curve}\xyoption{frame}

\def\la{\Lambda}
\def\ga{\Gamma}
\def\lk{\ell_k}

\def\NN{\mathbb{N}}
\def\RR{\mathbb{R}}

\def\QQ{\mathbb{Q}}
\def\CC{\mathbb{C}}

\def\RRp{{\RR^+}}

\def\th{^\textrm{th}}
\def\Gal{\operatorname{Gal}}
\def\Om{\operatorname{\Omega}}

\def\mod{\operatorname{mod}}
\def\smod{\operatorname{\underline{mod}}}
\def\Db#1{{D^b(#1)}}

\def\Dbst#1{{D^b_{\rm st}(#1)}}

\def\cx#1{\operatorname{cx_{#1}}}
\def\cxl{\cx{\la}}

\newtheorem{thm}{Theorem}[section]

\theoremstyle{plain}
\newtheorem{FinSyzComp}[thm]{Theorem}
\newtheorem{DbstEquiv2}[thm]{Theorem}

\theoremstyle{plain}
\newtheorem{LiftLower}[thm]{Lemma}

\theoremstyle{plain}
\newtheorem{ExactCx}[thm]{Corollary}
\newtheorem{DbstEquiv}[thm]{Corollary}

\theoremstyle{plain}
\newtheorem{ConvWellDef}[thm]{Proposition}
\newtheorem{ConvPolyExp}[thm]{Proposition}
\newtheorem{LowerBoundCx}[thm]{Proposition}
\newtheorem{StrongQuiv}[thm]{Proposition}
\newtheorem{QuiverComplexity}[thm]{Proposition}
\newtheorem{Curvatures}[thm]{Proposition}
\newtheorem{Partial}[thm]{Proposition}
\newtheorem{Partial2}[thm]{Proposition}
\newtheorem{Partial3}[thm]{Proposition}

\newtheorem{Sinkfree}[thm]{Proposition}

\theoremstyle{definition}
\newtheorem{Fibonacci}[thm]{Example}
\newtheorem{VanRadSqu}[thm]{Example}
\newtheorem{StrongRep1}[thm]{Example}
\newtheorem{StrongRep2}[thm]{Example}
\newtheorem{StrongRep3}[thm]{Example}

\begin{document}

\address{department of mathematics, university of california, santa barbara, ca, 93106-3080}
\email{thoward@math.ucsb.edu}

\title[Poly-exponential complexities]{Representations of monomial algebras have poly-exponential complexities}
\author{Tom Howard}

\begin{abstract}
We use directed graphs called ``syzygy quivers'' to study the asymptotic growth rates of the dimensions of the syzygies of representations of finite dimensional algebras.  For any finitely generated representation of a monomial algebra, we show that this growth rate is poly-exponential, i.e. the product of a polynomial and an exponential function, and give a procedure for computing the corresponding degree and base from a syzygy quiver.  We characterize the growth rates arising in this context: The bases of the occurring exponential functions are the real, nonnegative algebraic integers $b$ whose irreducible polynomial over $\QQ$ has no root with with modulus larger than $b$.  Moreover, we show that these growth rates are invariant under stable derived equivalences.
\end{abstract}

\maketitle

\section{Introduction}

A {\sl monomial algebra} is a finite dimensional algebra $\la$ over a field $k$ with a presentation $\la = kQ/I$, where $Q$ is a finite quiver and $I$ an admissible ideal of $kQ$ generated by paths.  As a tool to study finitistic dimensions of such algebras, C. Cibils introduced syzygy quivers, combinatorial devices which encode the syzygies of certain modules \cite{Cibils}.  We refine this approach in order to explicitly compute complexities over monomial algebras and algebras which are derived equivalent to monomial algebras.

The ``complexity'' of a finitely generated module $A$ over a finite dimensional algebra $\la$ measures the dimension growth of the syzygies of $A$.  Various numerical measures of this type have been studied in the literature, including the Alperin-Carlson complexity (\cite{Alperin}, \cite{CarlsonComp}) and L. Avramov's curvature \cite{Avramov}.  The Alperin-Carlson complexity was first introduced in the context of group algebras and has since been studied primarily for modules over selfinjective algebras (see \cite{Alperin}, \cite{CarlsonComp}, \cite{SuppVarHoch}, \cite{SuppVar}, \cite{Bergh}, \cite{BerghOppermann}, \cite{DanEd}, \cite{DanOtto}) .  Our notion of complexity sharpens the previously considered ones; it can be consider a hybrid of the Alperin-Carlson complexity and L. Avramov's curvature, at least when $\la$ is a monomial algebra.  More precisely, for a finitely generated module over a monomial algebra, we show that the sequence of dimensions of its syzygies grows poly-exponentially, i.e. grows as the product of an exponential function and a polynomial.  The base of the exponential function is the curvature of the module in the sense of \cite{Avramov}.  When the curvature is one, the degree of the polynomial is one less than the Alperin-Carlson complexity of the module.  For larger curvatures, the degree of the polynomial is a new homological invariant.  We use syzygy quivers to explicitly compute each of these two parameters.  We also show that a real number is the curvature of some finitely generated module (over some monomial algebra) if and only if it is a nonnegative algebraic integer and no root of its irreducible polynomial is larger than it in modulus.

Monomial algebras are particularly amenable to study by syzygy quivers because each finitely generated module $A$ has {\sl finite syzygy type}, meaning that there is a finite list of modules $A_1, A_2, \ldots, A_r$ with the property that each syzygy of $A$ is a direct sum of copies of the $A_i$ \cite{BirgeKen}.  Indeed, for a monomial algebra $\la = kQ/I$, B. Huisgen-Zimmermann has shown in \cite{PredictingSyzygies} that each submodule of a projective $\la$-module is isomorphic to a direct sum of principal left ideals generated by paths of positive length in $kQ\backslash I$, so each finitely generated $\la$-module has finite syzygy type.  The results we present here apply more broadly to artin algebras over which each finitely generated module has finite syzygy type (or to individual modules with finite syzygy type).

R.-O. Buchweitz has defined the {\sl stable derived category} of a module category to be the Verdier quotient of its derived category by the thick subcategory of perfect complexes \cite{Buchweitz}.  We say that two artin algebras are {\sl stably derived equivalent} if the stable derived categories of their respective module categories are equivalent as triangulated categories.  In particular, two artin algebras are stably derived equivalent if they are derived equivalent.  We use syzygy quivers to show that if each finitely generated module over an artin algebra $\la$ has finite syzygy type, and $\ga$ is an artin algebra which is stably derived equivalent to $\la$, then each finitely generated $\ga$-module has finite syzygy type, and the spectrum of complexities of finitely generated $\la$-modules matches the spectrum of complexities of finitely generated $\ga$-modules.  This result is a useful and practical tool towards showing that a given algebra is not derived equivalent to a given monomial algebra.  See \cite{StrongTilt} for some examples of algebras which are derived equivalent to monomial algebras.  See also \cite{BerghOppermann} for some connections between representation dimension, complexity, and the stable derived categories of certain Gorenstein algebras.

If $A$ is a module that fails to have finite syzygy type, we will show ``partial syzygy quivers'' to still be useful in finding lower bounds for the complexity of $A$.  These lower bounds can often be transported through stable derived equivalences.

I thank my PhD advisor Birge Huisgen-Zimmermann for fruitful ideas, suggestions, and conversations.  I would also like to thank Rob Ackermann, Darren Long, Charles Martin, and Jon McCammond for their insights and for directing my attention to various resources.

\section{Complexity of a $\la$-module}

Throughout, $\la$ is an artin $k$-algebra over a commutative artinian ring $k$ with Jacobson radical $J$.  All $\la$-modules we consider will be finitely generated left modules, and we will denote by $\mod \la$ the category of finitely generated left $\la$-modules.  

The complexity of a $\la$-module is a measure of the rate of growth of its syzygies.  We write $\Om A$ for the first syzygy of $A$ (i.e., the kernel of a projective cover of $A$), and put $\Om^0 A = A$ and $\Om^{n+1} A = \Om(\Om^n A)$ for $n \ge 0$.  The {\sl complexity} of $A$, written $\cxl(A)$, is defined to be the complexity class of the sequence $\lk \Om^n A$, where $\lk$ denotes the composition length over $k$, and the complexity class of a sequence is as follows.

Let $f:\NN \to \RRp$ be a sequence of nonnegative real numbers.  The {\sl complexity class} of $f$, written $[f(n)]$, is the collection of sequences $g:\NN \to \RRp$ for which there are real numbers $c_1, c_2 > 0$ with $c_1 g(n) \le f(n) \le c_2 g(n)$ for all but finitely many $n \in \NN$.  We obtain a partial ordering by setting $[f(n)] \le [g(n)]$ if there is a real number $c > 0$ with $f(n) \le c g(n)$ for all but finitely many $n \in \NN$.  Moreover, each pair of complexity classes $[f(n)]$ and $[g(n)]$ has a least upper bound with respect to this partial ordering, given by $[f(n)] + [g(n)] = [(f+g)(n)]$.

Many of our examples will be finite dimensional path algebras modulo relations, i.e. will be of the form $\la = kQ/I$, where $k$ is field, $Q$ a finite quiver, and $I$ an admissible ideal in $kQ$.  For each vertex $v$, we let $e_v$ be the corresponding primitive idempotent of $\la$, and $S_v = \la e_v / J e_v$ the corresponding simple module.  If $I$ may be generated by paths in $Q$, then $\la$ is a monomial algebra.

\begin{Fibonacci}
\label{Fibonacci}
Let $k$ be a field, $Q$ be the quiver 
$$\xymatrix{
1 \ar@/^/[r]<0.75ex>^\alpha & 2 \ar@/^/[l]<0.75ex>^\beta \ar@(ur,dr)^\gamma
}$$
and $\la = kQ/\langle \textrm{all paths of length 2} \rangle$.  Observe that $\Om^{n+2} S_1 \simeq \Om^{n+1} S_1 \oplus \Om^n S_1$ for all $n \ge 0$, while both $\Om^0 S_1$ and $\Om^1 S_1$ are simple.  Computing dimensions, we find that $\dim_k \Om^n S_1$ is the $n\th$ Fibonacci number, so $\cxl(S_1) = [\phi^n]$, where $\phi = \frac{1+\sqrt{5}}{2}$ is the golden ratio.
\end{Fibonacci}

Clearly, $\cxl(A) = [0]$ if and only if $A$ has finite projective dimension.  It is also easy to see that for each $\la$-module $A$, there is a real number $b \ge 0$ with $\cxl(A) \le [b^n]$.  Indeed, for each $\la$-module $C$, there is a surjection $\la^c \to C$, where $c = \lk C$.  If we set $b = \lk \la$, we find that $\lk \Om C \le b \lk C$.  So $\lk \Om^n A \le b^n \lk A$.

The {\sl curvature} of a $\la$-module $A$ is $\kappa(A) = \inf\{b \in \RR \mid \cxl(A) \le [b^n] \}$ \cite{Avramov}.  For any sequence $f:\NN \to \RRp$ with $[f(n)] \le [b^n]$ for some $b \in \RR$, we define the curvature $\kappa(f)$ analogously, and say that $f$ has {\sl finite curvature}.  If $f$ has finite curvature, then its generating function $F(x) = \sum_{n=0}^{\infty} f(n) x^n$ has a radius of convergence equal to $1/\kappa(f)$.

In section \ref{SyzygyQuivers}, we show that any $\la$-module $A$ with finite syzygy type has poly-exponential complexity, i.e. $\cxl(A) = [b^n n^\ell]$ for some real number $b = \kappa(A) \ge 0$ and integer $\ell \ge 0$.  A $\la$-module $A$ has {\sl finite syzygy type} if there is a finite list of $\la$-modules $A_1, A_2, \ldots, A_r$ such that $\Om^n A$ is a direct sum of copies of the $A_i$ for each integer $n \ge 0$.  Every module over a monomial algebra $\la$ has finite syzygy type since, as shown in \cite{PredictingSyzygies}, each submodule of a projective module is isomorphic to a direct sum of principal left ideals generated by paths of positive length in $kQ\backslash I$.

To prove that modules with finite syzygy type have poly-exponential complexities, we will make use of the discrete convolution of sequences and a condition on complexity classes which we call translation invariance.  A complexity class $[f(n)]$ is {\sl translation invariant} if $[f(n+1)] = [f(n)]$.  One may check that this condition is independent of one's choice of representative from the complexity class.  Poly-exponential complexity classes are translation invariant.

The {\sl discrete convolution}, or {\sl Cauchy product}, of two sequences $f,g:\NN \to \RRp$ is given by $(f*g)(n) = \sum_{s=0}^n f(s)g(n-s)$.  If $F(x)$ and $G(x)$ are the generating functions of $f$ and $g$ respectively, $F(x)G(x)$ is the generating function of $f*g$.  It follows that $*$ is associative, commutative, distributes across addition, and has an identity element $\varepsilon$: the sequence with generating function $1$.  In general, the convolution of two complexity classes $[f(n)]*[g(n)] = [(f*g)(n)]$ is not well-defined.  Indeed, $[0] = [\varepsilon]$, while $0*f = 0$ and $\varepsilon*f = f$ for any sequence $f$.  Additionally, if $\tau$ is the sequence with generating function $x$, then $[\varepsilon] = [\tau]$, yet $[(\tau*f)(n)] = [f(n)]$ if and only if $[f(n)]$ is translation invariant.  This suggests that we restrict our attention to translation invariant complexity classes and avoid the zero sequence.  The next proposition allows us to unambiguously define the convolution of two translation invariant complexity classes by $[f(n)]*[g(n)] = [(f*g)(n)]$, with the caveat that we refrain from choosing the zero sequence when working with the complexity class $[0]$.  Moreover, the proposition guarantees that the resulting complexity class is itself translation invariant, and that convolution with a fixed translation invariant complexity class preserves the ordering on complexity classes.

\begin{ConvWellDef}
\label{ConvWellDef}
Let $f, f', g, g', h:\NN \to \RRp$, none identically zero, $[f(n)] = [f'(n)]$, $[g(n)] = [g'(n)]$, $[h(n)] \le [f(n)]$, and each of these complexity classes be translation invariant.  Then $f*g$ is not identically zero, $[(f*g)(n)] = [(f'*g')(n)]$, $[(h*g)(n)] \le [(f*g)(n)]$, and $[(f*g)(n)]$ is translation invariant.
\end{ConvWellDef}
\begin{proof}
Since $f$ and $g$ are not identically zero, there exist $s, t \ge 0$ with $f(s)>0$ and $g(t)>0$.  Then $(f*g)(s+t) \ge f(s)g(t) > 0$, so $f*g$ is not identically zero.

As $[f(n)] = [f'(n)]$ and $[g(n)] = [g'(n)]$, there are real numbers $A,B,C,D>0$ and an integer $N>s,t$ with $Af'(n) \le f(n) \le Bf'(n)$ and $Cg'(n) \le g(n) \le Dg'(n)$ for all $n \ge N$.  Now, for $n \ge 2N$, $$(f*g)(n) = \sum_{i=0}^{N-1} f(i)g(n-i) + \sum_{i=N}^{n-N}f(i)g(n-i) + \sum_{i=(n-N)+1}^n f(i)g(n-i).$$  Using translation invariance of $[f(n)]$ and $[g(n)]$ together with the requirement that $N > s,t$ we see that $[(f*g)(n)] = [f(n)] + [g(n)] + [\sum_{i=N}^{n-N}f(i)g(n-i)]$.  By our choice of $N$, we find $[\sum_{i=N}^{n-N}f(i)g(n-i)] = [\sum_{i=N}^{n-N}f'(i)g'(n-i)]$, so $[(f*g)(n)] = [(f'*g')(n)]$.

Since $[h(n)] \le [f(n)]$, we may choose $h' \in [h(n)]$, not identically zero, with $h'(n) \le f(n)$ for all $n \ge 0$.  Thus $[(h*g)(n)] = [(h'*g)(n)] \le [(f*g)(n)]$.

To establish translation invariance, let $\phi(n) = f(n+1)$, and observe that $(f*g)(n+1) = f(0)g(n+1) + (\phi*g)(n)$.  Since $[g(n)] = [(\varepsilon*g)(n)] \le [(f*g)(n)]$, we have $[(f*g)(n+1)] = [(\phi*g)(n)] = [(f*g)(n)]$.
\end{proof}

Now that the discrete convolution of translation invariant complexity classes is on a solid footing, we consider the discrete convolution of poly-exponential complexity classes.  The rule of thumb is that if one complexity class has larger curvature than the other, then it dominates and the term with lower curvature is ignored.  If the two curvatures are equal, the curvature of the convolution is kept and the sum of the degrees of the pertinent polynomials is increased increased by one.

\begin{ConvPolyExp}
\label{ConvPolyExp}
Given real numbers $a ,b \ge 0$ and integers $r,\ell \ge 0$, $$\left[a^n n^r\right] * \left[b^n n^\ell \right] =
\begin{cases}
\left[b^n n^\ell \right] & a < b \\
\left[b^n n^{\ell + r + 1} \right] & a = b \\
\left[a^n n^r \right] & a > b
\end{cases}$$
\end{ConvPolyExp}
\begin{proof}
If either $a=0$ or $b=0$, the result holds trivially.  Now suppose $a,b>0$.  Observe that for any integer $s \ge 0$, we have $[b^n n^s]*[b^n] = [b^n \sum_{i=0}^{n} i^s] = [b^n n^{s+1}]$.  This allows us to write $[b^n n^\ell]$ as the convolution of $\ell+1$ copies of $[b^n]$, and similarly for $[a^n n^r]$.  Next observe that if $a < b$, then $[a^n]*[b^n] = [b^n \sum_{i=0}^n (a/b)^i] = [b^n]$.  By a similar argument, $[a^n]*[b^n] = [a^n]$ when $b < a$.  In either case, we arrive at our desired conclusion using the commutativity and associativity of convolution.
\end{proof}

\section{Syzygy quivers in $\mod \la$}
\label{SyzygyQuivers}

A {\sl syzygy quiver} in $\mod \la$ is a pair $(Q,\varphi)$, where $Q$ is a quiver and $\varphi$ assigns to each vertex $v$ of $Q$ a $\la$-module $\varphi(v) \neq 0$ such that $\Om \varphi(v) \simeq \bigoplus_w \varphi(w)^{a_{vw}}$, where $a_{vw}$ is the number of arrows from $v$ to $w$ in $Q$.  We say that the module $\varphi(v)$ {\sl lies on} the syzygy quiver $(Q,\varphi)$ {\sl at} $v$.  C. Cibils introduced certain syzygy quivers for which $\varphi(v)$ is always indecomposable in \cite{Cibils}.  We waive this requirement because the categories we consider in the next section frequently have no indecomposable objects.

Every nonzero $\la$-module lies on at least one syzygy quiver, which may be chosen to be finite if and only if the module has finite syzygy type. If $(Q,\varphi)$ is a syzygy quiver in $\mod \la$, each vertex in $Q$ necessarily has finite out-degree.

\begin{VanRadSqu}
\label{VanRadSqu}
Suppose $k$ is a field and $Q$ is any finite quiver.  Let $\la = kQ/\langle$all paths of length $2 \rangle$.  Setting $\varphi(v) = S_v$ for each vertex, we find that $(Q,\varphi)$ is a syzygy quiver in $\mod \la$.
\end{VanRadSqu}

\begin{StrongRep1}
\label{StrongRep1}
Let $k$ be a field, and let $\la$ be the commutative algebra $k[X,Y,Z]/\langle X^2,Y^2,Z^2,XZ,YZ\rangle$.  There is a unique simple $\la$-module $k$, and a sequence of indecomposable modules $B_0(=k), B_1, B_2, \ldots$ satisfying $b_n = \dim_k B_n = 2n+1$ and $\Om B_n \simeq B_{n+1} \oplus k^{n+1}$ for all $n \ge 0$.  If we let $Q$ be the quiver

$$\xymatrix{
v_0 \ar@(ul,dl)_{a_0} \ar[r] & v_1 \ar@/_/[l]_{a_1} \ar[r] & v_2 \ar@/^1pc/[ll]^{a_2} \ar[r] & \ldots,
}$$

\noindent where $a_n = n+1$ is the number of arrows from $v_n$ to $v_0$ and put $\varphi(v_n) = B_n$ for each $n \ge 0$, then $(Q,\varphi)$ is a syzygy quiver.
\end{StrongRep1}

Syzygy quivers may be used to find lower bounds for the complexity of $\la$-modules.  If each vertex in a quiver $Q$ has finite out-degree, we define the complexity of a vertex $v$ to be $\cx{Q}(v) = [f(n)]$, where $f(n)$ is the number of paths of length $n$ in $Q$ beginning at $v$.

\begin{StrongRep2}
\label{StrongRep2}
We return to Example \ref{StrongRep1} and compute both $\cx{Q}(v_0)$ and $\cxl(k)$.  Let $f_s(n)$ be the number of paths of length $n$ from $v_0$ to $v_s$, and put $g = f_0$.  Then $f_s(n) = g(n-s)$.  It follows that the number of paths of length $n$ beginning at $v_0$ is $(1*g)(n)$ and that $\dim_k \Om^n k = (b*g)(n)$.

Let $G(x) = \sum_{n=0}^{\infty} g(n) x^n$ and $A(x) = \sum_{n=0}^{\infty} a_n x^n$.  Since $g(n+1) = a_0 f_0(n) + \ldots + a_n f_n(n) = (a*g)(n)$, we have $G(x) - 1 = x G(x) A(x)$, so $G(x) = \frac{1}{1-xA(x)}$.  But $a_n = n+1$, so $G(x) = \frac{(1-x)^2}{x^2 - 3x +1}$.  We thereby have $[g(n)] = \left[ \left(\frac{3 + \sqrt{5}}{2}\right)^n\right]$, and so since $[b_n] = [n]$, $\cxl(A) = \cx{Q}(v_0) = \left[ \left(\frac{3 + \sqrt{5}}{2}\right)^n\right]$.
\end{StrongRep2}

\begin{LowerBoundCx}
\label{LowerBoundCx}
If a $\la$-module $A$ lies on a syzygy quiver $(Q,\varphi)$ at a vertex $v$, then $\cxl(A) \ge \cx{Q}(v)$.  If in addition the modules lying on $(Q,\varphi)$ have bounded composition length, then $\cxl(A) = \cx{Q}(v)$.
\end{LowerBoundCx}
\begin{proof}
For each integer $n\ge0$, $\Om^n A \simeq \bigoplus_w \varphi(w)^{b_{vw}}$, where $b_{vw}$ is the number of paths of length $n$ in $Q$ from $v$ to $w$.  Each module lying on $(Q,\varphi)$ has composition length at least one, so $\cxl(A) \le \cx{Q}(v)$.  The latter claim is easy.
\end{proof}

\begin{ExactCx}
\label{ExactCx}
If a $\la$-module $A$ lies on a finite syzygy quiver $(Q,\varphi)$ at a vertex $v$, then $\cxl(A) = \cx{Q}(v)$.
\end{ExactCx}

This corollary reduces the problem of computing complexities of modules with finite syzygy type to that of computing complexities of vertices in finite quivers.  We accomplish this task by studying the strongly connected components of our finite quiver.  A quiver is {\sl strongly connected} if, for each ordered pair of vertices, there is a path connecting the first vertex to the second.  A {\sl strongly connected component} of a quiver is a maximal strongly connected subquiver.  If $Q$ is a quiver, we let $Q^*$ denote the poset of strongly connected components of $Q$, with $H \le K$ is there is a path beginning at a vertex in $K$ and ending at a vertex in $H$.  Consider a finite quiver $Q$.  The existence of a uniform bound on the out-degree of the vertices ensures that $\cx{Q}(v) \le \cx{Q}(w)$ whenever there is a path from $w$ to $v$.  If $H$ is a strongly connected component of $Q$, then each vertex of $H$ has the same complexity in $Q$.  We define the {\sl complexity of $H$ in $Q$} to be $\cx{Q}(H) = [f(n)]$, where $f(n)$ is the number of paths of length $n$ in $Q$ beginning at a vertex in $H$, and observe that for each vertex $v$ in $H$, $\cx{Q}(H) = \cx{Q}(v)$.  The task of computing the complexity of a vertex in a finite quiver now reduces to the task of computing the complexity of its strongly connected component.  

To compute complexities in a finite quiver $Q$, we will need both the poset $Q^*$ and some numerical information reflecting the internal structure of the strongly connected components.  For each strongly connected component $H$, we let $\rho(H)$ be the spectral radius of some adjacency matrix of $H$.  We will call $\rho(H)$ the {\sl spectral radius of $H$}.

The following two propositions allow us to compute complexities of vertices first in a strongly connected finite quiver, and then in an arbitrary finite quiver.  We adopt the following notation for convenience.  If $H$ and $K$ are subquivers of a finite quiver $Q$, then $p(H,K,n)$ is the number of paths of length $n$ in $Q$ beginning at a vertex in $H$ and ending at a vertex in $K$.  If $\beta$ is an arrow in $Q$, then $p(H,\beta,K,n)$ will count only those paths which include the arrow $\beta$.  For purposes of this function $p$, we identify each vertex with the subquiver consisting of only this vertex and no arrows.  For instance, $\cx{Q}(v) = [p(v,Q,n)]$ and $\cx{Q}(H) = [p(H,Q,n)]$.

\begin{StrongQuiv}
\label{StrongQuiv}
Let $H$ be a strongly connected finite quiver with $\rho = \rho(H)$.  Then $\cx{H}(H) = [\rho^n]$.
\end{StrongQuiv}
\begin{proof}
If $H$ has no arrows, then $\rho = 0$ and $\cx{H}(H) = [0]$, so the proposition holds.  Otherwise, order the vertices $v_1, \ldots, v_m$ of $H$ and let $A$ be the adjacency matrix of $H$ with respect to this ordering.  Then $A$ is an irreducible nonnegative matrix, and by the Perron-Frobenius theorem $\rho$ is a positive real eigenvalue of $A$ with a positive eigenvector \[ x = \left( \begin{array}{c}
x_1 \\
x_2 \\
\vdots \\
x_m \end{array} \right).\] 
Then $(A^n)_{ij} = p(v_i,v_j,n)$, so $$p(H,H,n) = \sum_{0 \le i,j \le m} (A^n)_{ij}.$$  Let $a = \min\{x_1, \ldots, x_m\} > 0$ and $b = \max \{x_1, \ldots, x_m \} > 0$.
Now, $$a \sum_{0 \le i,j \le m} (A^n)_{ij} \le \sum_{0 \le i,j \le m} (A^n)_{ij} x_j \le b \sum_{0 \le i,j \le m} (A^n)_{ij}.$$
But, $$\sum_{0 \le i,j \le m} (A^n)_{ij} x_j = \rho^n \sum_{0 \le i \le m} x_i.$$
Therefore $\cx{H}(H) = [\rho^n]$ as desired.
\end{proof}

\begin{QuiverComplexity}
\label{QuiverComplexity}
Let $Q$ be a finite quiver.  Given a strongly connected component $H \in Q^*$, let $b = \max \{ \rho(K) \mid K \le H \}$ and $\ell$ be maximal such that there is a chain $K_0 < K_1 < \ldots < K_\ell \le H$ in $Q^*$ with $\rho(K_i) = b$ for $0 \le i \le \ell$.  Then $\cx{Q}(H) = [b^n n^\ell]$.
\end{QuiverComplexity}
\begin{proof}
Let $H \in Q^*$ with $\rho = \rho(H)$, and let $b$ and $\ell$ be as above.  If $H$ is minimal with respect to $\le$,  then $b = \rho$, $\ell = 0$, and $\cx{Q}(H) = \cx{H}(H) = [\rho^n]$ as desired.  Now suppose $H$ is not minimal, but that the proposition holds for each $K < H$.

Given an arrow $\alpha$, let $s \alpha$ be the starting vertex and $t \alpha$ be the ending vertex.  Let $\mathcal B$ be the set of arrows $\beta$ with $s \beta$ in $H$ and $t \beta$ not in $H$.  Each path beginning at a vertex in $H$ either remains in $H$ or includes precisely one arrow from $\mathcal B$.  So $p(H,Q,n) = p(H,H,n) + \sum_{\beta \in \mathcal B} p(H,\beta,Q,n)$.

Observe that $p(H,\beta,Q,n) = \sum_{i=0}^{n-1} p(H,s\beta,i)p(t\beta,Q,n-1-i)$.  Let $f(n) = p(H,s\beta,n)$ and $g_\beta(n) = p(t\beta,Q,n)$.  Then $p(H,\beta,Q,n) = (f*g_\beta)(n-1)$ for $n > 0$.  It follows that $\cx{Q}(H) = [f(n)]*\sum_{\beta \in \mathcal B} [g_\beta(n)]$.  By Proposition \ref{StrongQuiv}, we know $[f(n)] = [\rho^n]$.

We now have three cases.  First, if $\rho=b$ and $\ell=0$, then $\rho(K) < b$ for each $K < H$ and $\sum_{\beta \in \mathcal B} [g_\beta(n)] = [a^n n^s]$ for some $a < b$.  But then $\cx{Q}(H) = [b^n]*[a^n n^s] = [b^n]$.

Second, if $\rho = b$ and $\ell > 0$, then $\sum_{\beta \in \mathcal B} [g_\beta(n)] = [b^n n^{\ell-1}]$.  Thus $\cx{Q}(H) = [b^n]*[b^n n^{\ell-1}] = [b^n n^\ell].$

Last, if $\rho < b$, then $\sum_{\beta \in \mathcal B}[g_\beta(n)] = [b^n n^\ell]$.  So $\cx{Q}(H) = [\rho^n]*[b^n n^\ell] = [b^n n^\ell].$
\end{proof}

We have now established that the complexity of any $\la$-module with finite syzygy type is poly-exponential, and is therefore pinned down by two parameters:  the $b$ and $\ell$ of Proposition \ref{QuiverComplexity}.  The parameter $b$ is the curvature of the module.  When the curvature is zero, $\ell$ is the projective dimension.  When the curvature is one, $\ell$ is related to the complexity in the sense of Alperin.  Example \ref{VanRadSqu} shows that every finite quiver may be realized as a syzygy quiver in $\mod \la$ for some finite dimensional algebra $\la$ over a field.  A number may therefore be realized as the curvature of some module with finite syzygy type over some artin algebra $\la$ if and only if it is the spectral radius of a strongly connected finite quiver.

\begin{Curvatures}
\label{Curvatures}
The following are equivalent for a real number $b$.
\item{\rm (a)} There exists a strongly connected finite quiver with spectral radius $b$.
\item{\rm (b)} There exists an integer matrix with spectral radius $b$.
\item{\rm (c)} $b$ is a nonnegative algebraic integer and no root of its irreducible polynomial over $\QQ$ is larger than it in modulus.

Moreover, the set of real numbers satisfying the above conditions is closed under addition, multiplication, and passage to roots.  This set is therefore a dense subset of $\{0\}\cup[1,\infty)$.
\end{Curvatures}
\begin{proof}
First we show that the set of numbers satisfying (c) has the desired closure properties.  Let $b$ satisfy condition (c) with irreducible polynomial $p(x)$, say of degree $m$.  Given an integer $\ell \ge 0$, $\sqrt[\ell]{b}$ satisfies condition (c), since its irreducible polynomial divides $p(x^\ell)$.  Now let $c \neq 0$ be another number satifying condition (c) whose irreducible polynomial $q(x)$ has roots $a_1, a_2, \ldots, a_n$.  Consider the polynomials $r(x) = p(x-a_1)p(x-a_2) \cdots p(x-a_n)$ and $s(x) = (-q(0))^m p(\frac{x}{a_1})p(\frac{x}{a_2}) \cdots p(\frac{x}{a_n})$.  These polynomials are each monic with integer coefficients, and they establish condition (c) for $b+c$ and $bc$, since the irreducible polynomials of $b+c$ and $bc$ divide $r(x)$ and $s(x)$ repsectively.

We now show ``(a) $\Rightarrow$ (b)''.  Let $Q$ be a strongly connected finite quiver with spectral radius $b$.  By definition, $b$ is the spectral radius of an adjacency matrix of $Q$, which is an integer matrix.

Next we show ``(b) $\Rightarrow$ (c)''.  Let $A$ be an integer matrix with spectral radius $b$.  The characteristic polynomial $q(x)$ of $A$ is a monic polynomial with integer coefficients.  For some complex number $\omega$ with unit length, both $b \omega$ and its complex conjugate $b \bar{\omega}$ are roots of $q(x)$.  Then $b$ is an algebraic integer, since it is a square root of $(b \omega) (b \bar{\omega})$.  Now let $p(x)$ be the irreducible polynomial of $b$ over $\QQ$, with splitting field $L$ over $\QQ$.  Let $a$ be another root of $p(x)$, and choose $\varphi \in \Gal(L/\QQ)$ with $\varphi(b) = a$.  By hypothesis on $q(x)$, both $|\varphi(b\omega)| \le b$ and $|\varphi(b \bar{\omega})| \le b$.  Since either $|\varphi(\omega)| \ge 1$ or $|\varphi(\bar{\omega})| \ge 1$, we conclude $|a| \le b$.

Finally, we show ``(c) $\Rightarrow$ (a)''.  Let $p(x)$ be the monic irreducible polynomial of $b$ over $\QQ$.  Up to sign, $p(0)$ is the product of the roots of $p(x)$, so either $b=0$ or $b\ge 1$.  Each of zero and one is the spectral radius of some finite strongly connected quiver, so we focus on the case $b>1$.

The polynomial $p(x)$ has no repeated roots since it is irreducible over $\QQ$.  Let $m$ be the number of distinct roots of $p(x)$ with modulus equal to $b$, and let $\omega$ be a primitive $m\th$ root of unity.  Suppose that the set of roots of $p(x)$ is closed under multiplication by $\omega$, and so $p(x)=q(x^m)$ for some polynomial $q(x)$.  Then $q(x)$ has a nonnegative real root $\sqrt[m]{b}$ which is strictly larger than the modulus of any other root.  This property of $q(x)$ ensures that any integer matrix $M$ with characteristic polynomial $q(x)$ has a power $M^n$, with $n \ge 0$, which is conjugate to a matrix with only positive integer entries \cite{Ackermann}.  So $M^{mn}$ is conjugate to the adjacency matrix of a strongly connected quiver $Q$ with spectral radius $b^n$.  By replacing each arrow in $Q$ with a path of length $n$, we arrive at a finite strongly connected quiver with spectral radius $b$.

We now justify our supposition that the set of roots of $p(x)$ is closed under multiplication by the primitive $m\th$ root of unity $\omega$.  We let $B$ be the set of roots of $p(x)$ with modulus equal to $b$, and let $L$ be the splitting field of $p(x)$ over $\QQ$.  We first show that each automorphism $\varphi \in \Gal(L/\QQ)$ which sends $b$ to a root in $B$ must in fact permute the roots in $B$.  Indeed, suppose $\varphi \in \Gal(L/\QQ)$ with $\varphi(b) = b \eta \in B$ and let $b \nu \in B$.  We also have the complex conjugate $b \bar{\nu} \in B$, and $\varphi(b \nu) \varphi(b \bar{\nu}) = \varphi(b)^2$ has modulus $b^2$.  No root of $p(x)$ has modulus larger than $b$, so we conclude that $|\varphi(b \nu)| = b$.  Thus $\varphi$ indeed permutes the roots of $B$.  Given $\beta \in B$, one checks that $\varphi(\overline{\varphi^{-1}(\bar{\beta})}) = \beta \eta^2$.  Since both $b$ and $b \eta$ are in $B$, we have that $b \eta^i$ is a root for each $i \ge 0$.  Then $\eta$ must be a root of unity, else $p(x)$ would have infinitely many distinct roots.  Now let $a$ be any root of $p(x)$, and choose $\psi \in \Gal(L/\QQ)$ with $\psi(b) = a$.  Since $\eta$ is a root of unity, its powers are permuted by $\psi$, and there is some integer $i$ with $\psi(\eta^i) = \eta$.  Then $\psi(b \eta^i) = a \eta$.  This shows that the set of roots of $p(x)$ is closed under multiplication by $\eta$. Since $b \eta$ was an arbitrary root in $B$, this shows that the set of roots of $p(x)$ is closed under multiplication by $\omega$.
\end{proof}

We are now ready to present our main theorem on the complexity of a $\la$-module with finite syzygy type.

\begin{FinSyzComp}
\label{FinSyzComp}
Let $A$ be a $\la$-module with finite syzygy type.  Then $\cxl(A) = [b^n n^\ell]$ for some integer $\ell \ge 0$ and real number $b \ge 0$.  For each $0 \le s \le \ell$, there is an integer $m \ge 0$ and a direct summand $B$ of $\Om^m A$ with $\cxl(B) = [b^n n^s]$.

Conversely, given an integer $\ell \ge 0$ and real number $b \ge 0$, there is an artin algebra $\la$ and $\la$-module $A$ with finite syzygy type and $\cxl(A) = [b^n n^\ell]$ if and only if $b$ is an algebraic integer and no root of its irreducible polynomial over $\QQ$ has modulus greater than $b$.
\end{FinSyzComp}
\begin{proof}
If $A = 0$, then $\cxl(A) = [0]$ and the first half of the theorem holds trivially.  If $A$ is nonzero, it lies on some finite syzygy quiver $(Q,\varphi)$ at a vertex $v$.  Let $H$ be the strongly connected component of $Q$ containing $v$.  Then $\cxl(A) = \cx{Q}(H) = [b^n n^\ell]$, with $\ell$ and $b$ given by Proposition \ref{QuiverComplexity}.  Choose a chain $K_0 < \ldots < K_\ell \le H$  with $\rho(K_i) = b$ for each $0 \le i \le \ell$.  Given an integer $s \in \{0, \ldots, \ell \}$, choose a vertex $w \in K_s$ and a path from $v$ to $w$, say of length $m$.  Then $B = \varphi(w)$ is isomorphic to a direct summand of $\Om^m A$ with $\cxl(B) = [b^n n^s]$.

Now suppose we are given an integer $\ell \ge 0$ and a real nonnegative algebraic integer $b$, and that no root of the irreducible polynomial of $b$ has modulus greater than $b$.  Proposition \ref{Curvatures} ensures that there is a finite strongly connected quiver $H$ with spectral radius $b$.  Let $P_\ell$ be the quiver $$\ell \to \ldots \to 1 \to 0,$$ and put $Q = H \times P_\ell$.  Let $k$ be a field and $\la = kQ/\langle$all paths of length 2$\rangle$.  For each vertex $v$ of $H$, the simple $\la$-module $S_{(v,s)}$ has finite syzygy type and $\cxl(S_{(v,s)}) = [b^n n^s]$.
\end{proof}

\section{The stable derived category of $\mod \la$}

There is a canonical embedding of $\mod \la$ into its bounded derived category $\Db{\mod \la}$.  The {\sl stable derived category} of $\mod \la$, denoted $\Dbst{\mod \la}$, is the Verdier quotient of $\Db{\mod \la}$ by the thick subcategory generated by the projective $\la$-modules.  By composing functors, we obtain $\sigma_\la: \mod \la \to \Dbst{\mod \la}$, which we abbreviate as $\sigma$ when there is no confusion regarding the algebra in question.  Moreover, $\sigma$ factors uniquely through the stable category $\smod \la$, and the resulting functor $\smod \la \to \Dbst{\mod \la}$ is universal among all exact functors from $\smod \la$ to a triangulated category \cite{KellerVossieck}.  Here we consider $\smod \la$ as a left triangulated category with cosuspension $\Om$, and so we write the cosuspension on $\Dbst{\mod \la}$ as $\Om$ as well.  We say that two artin algebras $\la$ and $\ga$ are {\sl stably derived equivalent} if there is a triangle equivalence from $\Dbst{\mod \la}$ to $\Dbst{\mod \ga}$.  For example, $\la$ and $\ga$ are stably derived equivalent if $\la$ and $\ga$ are derived equivalent, or if $\la$ and $\ga$ are stably equivalent via a left triangle equivalence.

Syzygy quivers can be used to transport information through stable derived equivalences, particular when dealing with modules with finite syzygy type.  We will need only a couple basic facts regarding the stable derived category of $\mod \la$.  For each object $X$ in $\Dbst{\mod \la}$, there is a $\la$-module $A$ such that $\sigma A \simeq \Om^n X$ for some integer $n \ge 0$.  Second, given a pair of $\la$-modules $A$ and $B$, we have $\sigma A \simeq \sigma B$ if and only if $\Om^n A \simeq \Om^n B$ for some $n \ge 0$.

A {\sl syzygy quiver} in $\Dbst{\mod \la}$ is a pair $(Q,\psi)$, where $Q$ is a quiver and $\psi$ assigns to each vertex $v$ of $Q$ a nonzero object $\psi(v) \in \Dbst{\mod \la}$ with $\Om \psi(v) \simeq \bigoplus_w \psi(w)^{a_{vw}}$ where $a_{vw}$ is the number of arrows in $Q$ from $v$ to $w$.

As with $\la$-modules, each nonzero object in $\Dbst{\mod \la}$ lies on some syzygy quiver.  Moreover, this syzygy quiver will always be sinkfree.  In fact, sinkfree syzygy quivers are also sufficient for computing complexities of $\la$-modules with finite syzygy type.

\begin{Sinkfree}
\label{Sinkfree}
Each $\la$-module with infinite projective dimension and finite syzygy type lies on a finite sinkfree syzygy quiver.
\end{Sinkfree}
\begin{proof}
Let $A$ be such a $\la$-module, and choose a finite syzygy quiver $(Q,\varphi)$ and vertex $v$ with $A = \varphi(v)$.  Construct a new quiver $Q'$ by deleting each sink from $Q$ and adding a new vertex $v'$ with a single arrow from $v'$ to $v$.  Then the assignment $\varphi'(v') = A$ and $\varphi'(w) = \Om \varphi(w)$ for each vertex $w$ of $Q$ makes $(Q',\varphi')$ a syzygy quiver.  By iterating this procedure, we reach a finite sinkfree syzygy quiver on which $A$ lies.
\end{proof}

Our main tool for transporting the information encoded in syzygy quivers through stable derived equivalences is the following lemma, which describes the extent to which syzygy quivers may be lowered from $\mod \la$ to $\Dbst{\la}$ and lifted back.

\begin{LiftLower}
\label{LiftLower}
\item{\rm (a)} If $(Q,\varphi)$ is a sinkfree syzygy quiver in $\mod \la$, then $(Q,\sigma \varphi)$ is a syzygy quiver in $\Dbst{\mod \la}$.
\item{\rm (b)} If $(Q,\psi)$ is a finite syzygy quiver in $\Dbst{\mod \la}$, there is a syzygy quiver $(Q,\varphi)$ in $\mod \la$ with $\sigma \varphi = \Om^n \psi$ for some integer $n \ge 0$.
\end{LiftLower}
\begin{proof}
\item{\rm (a)} The only potential obstruction would be $\sigma \varphi(v) = 0$ for some vertex $v$.  However, each module lying on $(Q,\varphi)$ has infinite projective dimension since $Q$ is sinkfree.
\item{\rm (b)} For each vertex $v$ in $Q$, let $X_v = \psi(v)$ and let $a_{vw}$ be the number of arrows from $v$ to $w$.  Then $\Om X_v \simeq \bigoplus_w X_w^{a_{vw}}$.  By choosing $n$ large enough, there exist $\la$-modules $A_v$ such that $\Om A_v \simeq \bigoplus_w A_w^{a_{vw}}$ and $\sigma A_v \simeq \Om^n X_v$ for each vertex $v$.  Then the assignment $\varphi(v) = A_v$ yields a syzygy quiver $(Q,\varphi)$ in $\mod \la$ with $\sigma \varphi = \Om^n \psi$.
\end{proof}

\begin{DbstEquiv}
\label{DbstEquiv}
Let $\la$ and $\ga$ be artin algebras and $F: \Dbst{\mod \la} \to \Dbst{\mod \ga}$ a triangle equivalence.  Suppose $A \in \mod \la$ and $B \in \mod \ga$ with $F \sigma_{\la} A = \sigma_{\ga} B$, and $A$ has finite syzygy type.  Then $B$ has finite syzygy type and $\cx{\ga}(B) = \cxl(A)$.
\end{DbstEquiv}
\begin{proof}
Suppose $A$ has finite syzygy type.  If $A$ has finite projective dimension, then $\sigma_{\ga} B = 0$, so $B$ has finite projective dimension too, and $\cxl(A) = \cx{\ga}(B) = [0]$.

If $A$ has infinite projective dimension, it lies on some finite sinkfree syzygy quiver $(Q,\varphi)$ at a vertex $v$.  Then $(Q,F\sigma_\la \varphi)$ is a syzygy quiver in $\Dbst{\mod \ga}$.  For some integer $n \ge 0$, there is a syzygy quiver $(Q, \psi)$ in $\mod \ga$ with $\psi(v) = \Om^n B$.  So $B$ has finite syzygy type and $\cx{\ga}(\Om^n B) = \cxl(A)$.  Since poly-exponential complexities are translation invariant, $\cx{\ga}(B) = \cxl(A)$.
\end{proof}

While the above corollary is useful whenever an artin algebra $\la$ has $\la$-modules with finite syzygy type and nontrivial complexity, we arrive at a particularly clean result when every $\la$-module has finite syzygy type.

\begin{DbstEquiv2}
\label{DbstEquiv2}
Let $\la$ and $\ga$ be artin algebras which are stably derived equivalent to one another.  Suppose each $\la$-module has finite syzygy type (e.g. $\la$ is a monomial algebra over a field).  Then each $\ga$-module has finite syzygy type, and $\{\cxl(A) \mid A \in \mod \la\} = \{\cx{\ga}(B) \mid B \in \mod \ga\}$.
\end{DbstEquiv2}
\begin{proof}
Let $F: \Dbst{\mod \la} \to \Dbst{\mod \ga}$ be a triangle equivalence and let $B \in \mod \ga$.  For some integer $n \ge 0$, we have an $A \in \mod \la$ with $F \sigma_\la A \simeq \sigma_{\ga} \Om^n B$.  By Corollary \ref{DbstEquiv}, $B$ has finite syzygy type and $\cx{\ga}(B) = \cxl(A)$.  So each $\ga$-module has finite syzygy type and $\{\cxl(A)\mid A \in \mod \la\} \supseteq \{\cx{\ga}(B)\mid B \in \mod \ga\}$.  The reverse inclusion follows from the symmetry of our set up.
\end{proof}

\section{Partial Syzygy Quivers}

Our results thus far on modules with finite syzygy type depend heavily on the existence of finite syzygy quivers.  A finite syzygy quiver has three advantages:  it encodes the precise value of the complexity of each module which lies on it; each such complexity may be computed using Proposition \ref{QuiverComplexity}; and, modulo sinks, a finite syzygy quiver may be shifted back and forth between $\mod \la$ and $\Dbst{\mod \la}$.  In this section we will consider subquivers of syzygy quivers.  We use these subquivers to obtain lower bounds for the complexities of modules lying on them.  Moreover, by choosing finite subquivers we can compute these lower bounds using Proposition \ref{QuiverComplexity}.  If we choose a sinkfree finite subquiver, we will see that it can be passed back and forth between $\mod \la$ and $\Dbst{\mod \la}$.  This allows us to use syzygy quivers in $\mod \la$ to establish lower bounds for complexities of modules in $\mod \ga$ when $\la$ and $\ga$ are stably derived equivalent, even when the modules under consideration do not have finite syzygy type.

More precisely, we say that a pair $(Q,\varphi)$ is a {\sl partial syzygy quiver} if there is a syzygy quiver $(Q',\varphi')$ such that $Q$ is a subquiver of $Q'$, and $\varphi$ the restriction of $\varphi'$ to the vertices of $Q$.  Such a syzygy quiver $(Q',\varphi')$ is called a {\sl completion} of the partial syzygy quiver $(Q,\varphi)$.

\begin{Partial}
\label{Partial}
\item{\rm (a)} Every nonzero $\la$-module lies on a finite partial syzygy quiver.
\item{\rm (b)} If a $\la$-module $A$ lies on a partial syzygy quiver $(Q,\varphi)$ at $v$, then $\cxl(A) \ge \cx{Q}(v)$.
\end{Partial}
\begin{proof}
Part (a) is clear.  For part (b), suppose $A$ lies on a partial syzygy quiver $(Q,\varphi)$ at a vertex $v$.  Choose a completion $(Q',\varphi')$ of $(Q,\varphi)$.  Since $Q$ is a subquiver of $Q'$, we have $\cx{Q}(v) \le \cx{Q'}(v) \le \cxl(A)$.
\end{proof}

Before we address the lifting and lowering of partial syzygy quivers between $\mod \la$ and $\Dbst{\mod \la}$, we give the following characterization of partial syzygy quivers which does not depend on a completion.

\begin{Partial2}
\label{Partial2}
Let $Q$ be a quiver, and for each pair of vertices $v$ and $w$, let $a_{vw}$ be the number of arrows from $v$ to $w$ in $Q$.  Then $(Q,\varphi)$ is a partial syzygy quiver if and only if $\bigoplus_{w} \varphi(w)^{a_{vw}}$ is isomorphic to a direct summand of $\Om \varphi(v)$ for each vertex $v$ in $Q$.
\end{Partial2}
\begin{proof}
Suppose $(Q,\varphi)$ is a partial syzygy quiver, and choose a completion $(Q',\varphi')$.  For each pair of vertices $v$ and $w$ in $Q'$, let $b_{vw}$ be the number of arrows from $v$ to $w$ in $Q'$.  Since $Q$ is a subquiver of $Q'$, $a_{vw} \le b_{vw}$ whenever $v$ and $w$ are in $Q$.  Thus $\bigoplus_{w \in Q} \varphi(w)^{a_{vw}}$ is isomorphic to a direct summand of $\Om \varphi(v)$ for each vertex $v$ in $Q$.

Conversely, suppose that for each vertex $v$ of $Q$, $\bigoplus_{w} \varphi(w)^{a_{vw}}$ is isomorphic to a direct summand of $\Om \varphi(v)$, and let $A(v)$ be a complement.  We now constuct a completion $(Q',\varphi')$ of $(Q,\varphi)$.  Let $V$ be the set of vertices $v$ in $Q$ for which $A(v)$ is nonzero.  For each $v \in V$, let $(Q_v, \varphi_v)$ be a syzygy quiver with $A(v)$ lying at a vertex $w_v$.  Now let $Q'$ be the disjoint union of $Q$ and the $Q_v$, together with a single arrow from $v$ to $w_v$ for each $v \in V$.  Finally, put $\varphi'(w)$ equal to $\varphi(w)$ if $w$ is a vertex of $Q$, or else equal to $\varphi_v(w)$ if $w$ is a vertex of $Q_v$ for some $v$ in $Q$.  Then $(Q', \varphi')$ is a syzygy quiver and a completion of $(Q,\varphi)$.
\end{proof}

\begin{Partial3}
\label{Partial3}
\item{\rm (a)} Let $(Q,\varphi)$ be a sinkfree partial syzygy quiver in $\mod \la$.  Then $(Q,\sigma \varphi)$ is a partial syzygy quiver in $\Dbst{\mod \la}$.
\item{\rm (b)} Let $(Q,\psi)$ be a finite partial syzygy quiver in $\Dbst{\mod \la}$.  Then there is a partial syzygy quiver $(Q,\varphi)$ in $\mod \la$ with $\sigma \varphi = \Om^n \psi$ for some integer $n \ge 0$.
\item{\rm (c)} Suppose $\la$ and $\ga$ are stably derived equivalent artin algebras and $F: \Dbst{\mod \la} \to \Dbst{\mod \ga}$ is a triangle equivalence. If $\sigma_{\ga} B = F \sigma_{\la} A$ and $A$ lies on a finite partial syzygy quiver $(Q,\varphi)$ at $v$, then $\cx{\ga}(B) \ge \cx{Q}(v)$.
\end{Partial3}
\begin{proof}
(a) Let $(Q,\varphi)$ be a sinkfree partial syzygy quiver in $\mod \la$.  Since $Q$ is sinkfree, $\sigma \varphi(v) \neq 0$ for each vertex $v$ in $Q$.  Since $\sigma$ distributes across finite direct sums, $(Q,\sigma \varphi)$ is a partial syzygy quiver in $\Dbst{\mod \la}$.

(b) Similar to Lemma \ref{LiftLower} (b), using the characterization of partial syzygy quivers given by Proposition \ref{Partial2}.

(c) Suppose $A$ lies on a finite sinkfree partial syzygy quiver $(Q,\varphi)$ at a vertex $v$.  Then $B$ lies on $(Q, F \sigma \varphi)$ at $v$, and we may lift this to a syzygy quiver $(Q,\psi)$ in $\mod \ga$ such that $\psi(v) \simeq \Om^m B$ for some integer $m \ge 0$.  Then $\cx{\ga}(\Om^m B) \ge \cx{Q}(v)$.  Since $Q$ is finite, $\cx{Q}(v)$ is translation invariant, and $\cx{\ga}(B) \ge \cx{Q}(v)$.
\end{proof}

We conclude by illustrating these techniques in our favorite example.

\begin{StrongRep3}
\label{StrongRep3}
We return again to the setting of Examples \ref{StrongRep1} and \ref{StrongRep2}.  Suppose $\ga$ is an artin algebra stably derived equivalent to $\la$.  Then there is a $\ga$-module $C$ for which $\sigma_\ga C$ lies on a syzygy quiver $(Q,\psi)$ at $v_0$.

For each integer $s \ge 0$, let $Q_s$ be the full subquiver of $Q$ on the vertices $v_0, v_1, \ldots, v_s$.  Note that each $Q_s$ is strongly connected, and let $r_s =  \rho(Q_s)$.  Then $\cx{\ga}(C) \ge [r_s^n]$ for each $s$.  In fact, we now show that $r_s$ converges to $\frac{3+\sqrt{5}}{2}$, thereby proving the curvature of $C$ over $\ga$ is at least the curvature of $k$ over $\la$.

Since the $Q_s$ are nested subquivers of $Q$, $r_s$ is a monotonically increasing sequence in $s$ bounded above by $\frac{3+\sqrt{5}}{2}$.  Let $r$ be the limit of the sequence  $r_s$.  The characteristic polynomial of an adjacency matrix of $Q_s$ is $p_s(x) = x^{s+1} - (a_0 x^s + a_1 x^{s-1} + \ldots + a_s)$.  Thus $r_s^{-1}$ is a root of $q_s(x)=1-x(a_0 + a_1x + \ldots + a_sx^s)$.  As $s \to \infty$, we see $q_s(x) \to q(x) = 1 - x A(x)$ uniformly on the closed disk of radius $r_1^{-1} < 1$ in $\CC$.  The $q_s$ are equicontinuous by the Arzel\`{a}-Ascoli theorem, which together with uniform convergence implies that $q(r^{-1}) = 0$.  Since $\frac{3-\sqrt{5}}{2}$ is the only zero of $q(x)$ in the unit disk, $r = \frac{3 + \sqrt{5}}{2}$.
\end{StrongRep3}

\bibliographystyle{plain}
\bibliography{ComplexityPaper}

\begin{thebibliography}{10}

\bibitem{Ackermann}
R.~Ackermann.
\newblock PhD thesis, UCSB.

\bibitem{Alperin}
J.L. Alperin.
\newblock Periodicity in groups.
\newblock {\em Ill. J. Math.}, 21:776--783, 1977.

\bibitem{Avramov}
L.~L. Avramov.
\newblock Infinite {F}ree {R}esolutions.
\newblock In {\em Six {L}ectures on {C}ommutative {A}lgebra}, volume 166 of
  {\em {P}rogress in {M}athematics}. Birkh{\"{a}}user Verlag, Basel Boston
  Berlin, 1991.

\bibitem{Bergh}
P.~Bergh.
\newblock Complexity and periodicity.
\newblock {\em Colloq. Math.}, 104(2):169--191, 2006.

\bibitem{BerghOppermann}
P.~Bergh, S.~Iyengar, H.~Krause, and S.~Oppermann.
\newblock Dimensions of triangulated categories via {K}oszul objects.
\newblock arXiv:0802.0952v3, 2009.

\bibitem{Buchweitz}
R.-O. Buchweitz.
\newblock Maximal {C}ohen-{M}acaulay modules and {T}ate-cohomology over
  {G}orenstein rings.
\newblock unpublished, 1987.

\bibitem{CarlsonComp}
J.~Carlson.
\newblock Complexity and {K}rull dimension.
\newblock In {\em Representations of {A}lgebras}, volume 903 of {\em {L}ecture
  {N}otes in {M}athematics}. Springer, Berlin Heidelberg New York, 1981.

\bibitem{Cibils}
C.~Cibils.
\newblock The syzygy quiver and the finitistic dimension.
\newblock {\em Comm. in Algebra}, 21:4167--4171, 1993.

\bibitem{StrongTilt}
A.~Dugas and B.~Huisgen-Zimmermann.
\newblock Strongly tilting truncated path algebras.
\newblock {\em manuscripta math.}
\newblock (to appear).

\bibitem{SuppVar}
K.~Erdmann, M.~Holloway, N.~Snashall, {\O}.~Solberg, and R.~Taillefer.
\newblock Support varieties for selfinjective algebras.
\newblock {\em K-Theory}, 33:67--87, 2004.

\bibitem{BirgeKen}
K.R. Goodearl and B.~Huisgen-Zimmermann.
\newblock Repetitive resolutions over classical orders and finite dimensional
  algebras.
\newblock {\em Canadian Mathematical Society Conference Preceedings},
  24:205--225, 1998.

\bibitem{DanEd}
E.~Green and D.~Zacharia.
\newblock On modules of finite complexity over selfinjective artin algebras.
\newblock To appear.

\bibitem{PredictingSyzygies}
B.~Huisgen-Zimmermann.
\newblock Predicting syzygies over monomial relation algebras.
\newblock {\em Manuscripta Math.}, 70:157--182, 1991.

\bibitem{KellerVossieck}
B.~Keller and D.~Vossieck.
\newblock Sous les cat\'{e}gories d\'{e}riv\'{e}es.
\newblock {\em C. R. Acad. Sci. Paris}, 305:225--228, 1987.

\bibitem{DanOtto}
O.~Kerner and D.~Zacharia.
\newblock Auslander-{R}eiten theory for modules of finite complexity over
  selfinjective algebras.

\bibitem{SuppVarHoch}
N.~Snashall and {\O}.~Solberg.
\newblock Support varieties and {H}ochschild cohomology rings.
\newblock {\em Proc. London Math. Soc.}, 3(88):705--732, 2004.

\end{thebibliography}

\end{document}